\documentclass
[12pt]
{article}
\usepackage{amsmath,amssymb,amsthm
,makeidx,fancyhdr
}
\usepackage{color}

\theoremstyle{definition}
\newtheorem{definition}{Definition}
\newtheorem{ddefinition}{Definition}

\theoremstyle{plain}
\newtheorem{theorem}[definition]{Theorem}
\newtheorem{proposition}[definition]{Proposition}
\newtheorem{lemma}[definition]{Lemma}

\newtheorem{remark}{Remark}

\newcommand{\fr}{{}^\frown}
\newcommand{\dec}{\parallel}
\newcommand{\name}{\dot}
\newcommand{\can}{\check}
\newcommand{\force}{\Vdash}

\newcommand{\la}{\langle}
\newcommand{\ra}{\rangle}
\newcommand{\elem}{\prec}
\newcommand{\uhr}{\upharpoonright}

\newcommand{\power}{\mathcal{P}}
\newcommand{\rar}{\rightarrow}
\newcommand{\lar}{\leftarrow}
\newcommand{\ol}{\overline}
\newcommand{\po}{\mathbb{P}}
\newcommand{\ola}{\overline{\alpha}}
\newcommand{\olk}{\overline{\kappa}}
\newcommand{\old}{\overline{\delta}}

\newcommand{\D}{\mathfrak{D}}
\newcommand{\R}{\mathfrak{R}}
\newcommand{\E}{\overline{E}}

\newcommand{\rad}{\po_{\E,\lambda}}
\newcommand{\K}{K}

\DeclareMathOperator{\GCH}{GCH}
\DeclareMathOperator{\cf}{cf}
\DeclareMathOperator{\pcf}{pcf}
\DeclareMathOperator{\cov}{cov}
\DeclareMathOperator{\pp}{pp}

\def\l{{\langle}}
\def\r{{\rangle}}

\title{On the Splitting Number at Regular Cardinals}
\author{Omer Ben-Neria and Moti Gitik\footnote{The second author was partially supported by ISF grant no. 58/14.}}
\date{\today}

\begin{document}
\maketitle
\begin{abstract}
Let $\kappa$, $\lambda$ be regular uncountable cardinals such that $\lambda > \kappa^+$ is not a successor of a singular cardinal of low cofinality. 
We construct a generic extension with $s(\kappa) = \lambda$ starting from a ground model in which  $o(\kappa) = \lambda$ and prove that assuming $\neg 0^{\P}$,
$s(\kappa) = \lambda$ implies 
that $o(\kappa) \geq \lambda$ in the core model.
\end{abstract}

\section{Introduction}
The splitting number is a cardinal invariant mostly known for its continuum version $s = s(\aleph_0)$.
Generalizations of this invariant to regular uncountable cardinals have been studied
mainly by S. Kamo, T. Miyamoto, M. Motoyoshi, T. Suzuki
  \cite{suzuki} and J. Zapletal \cite{Zapletal}.
  \\
  For a cardinal $\kappa$ and two sets $a,x\in [\kappa]^{\kappa}$ we say $x$ splits $a$ if both $a \setminus x$ and $a \cap x$ have cardinality $\kappa$. A family of sets $F \subset [\kappa]^\kappa$ is a splitting family if for all $a \in [\kappa]^\kappa$ there exists some $x \in F$ which splits $a$. The splitting number $s(\kappa)$ is the minimal cardinality of a splitting family $F \subset [\kappa]^{\kappa}$.
   \\M. Motoyoshi showed that for a regular uncountable cardinal $\kappa$, $s(\kappa) \geq \kappa$ if and only if $\kappa$ is inaccessible, and
   T. Suzuki \cite{suzuki} proved that $s(\kappa) \geq \kappa^{+}$ if and only if $\kappa$ is weakly compact.
 S. Kamo and T. Miyamoto independently showed
 how to force  $s(\kappa) \geq \kappa^{++}$ from the assumption of a $2^\kappa-$supercompact cardinal $\kappa$.
  J. Zapletal  \cite{Zapletal} proved that $s(\kappa) = \kappa^{++}$ implies there exists an inner model with a measurable cardinal $\kappa$ with $o(\kappa) = \kappa^{++}$.
  \\The question of whether the lower bound of Zapletal can be improved remained open.
    The purpose of the present paper is to answer it negatively and to show the following:

   \begin{theorem}\label{theorem - forcing}
   Let $\kappa$, $\lambda$ be regular uncountable cardinals such that $\lambda > \kappa^+$ is not a successor
   of a singular cardinal of cofinality $\leq \kappa$. Assuming $\GCH$, $o(\kappa) \geq \lambda$ implies
   there exists a forcing extension in which $\kappa$ is regular and $s(\kappa) = \lambda$. \footnote{We believe that it is possible to modify the proof of Theorem \ref{theorem - forcing} in a way that will include all regular cardinals $\lambda > \kappa^+$.}
   \end{theorem}

   \begin{theorem}\label{theorem - inner model}
   Let $\kappa$, $\lambda$ be regular uncountable cardinals such that $\lambda > \kappa^+$ is not a successor
   of a singular cardinal of cofinality $\leq \omega_1$. Assuming $\neg 0^{\P}$, 
   $s(\kappa) \geq \lambda$ implies there is an inner model in which $o(\kappa) \geq \lambda$. 
   \end{theorem}

  The following explains some ideas behind the forcing construction. The basic construction of S. Kamo, as sketched
  in \cite{Zapletal}, starts with a supercompact cardinal $\kappa$ and uses a $\kappa-$support iteration $\po_{\kappa^{++}} = \{\po_\alpha,Q_\alpha,\mid \alpha < \kappa^{++}\}$ 
  of generalized Mathias forcings $Q_\alpha =\po(U_\alpha)$ which adds a generating set $k_\alpha \subset \kappa$ to the $V^{\po_\alpha}$ measure $U_\alpha$ on $\kappa$, 
  i.e.  $k_\alpha \subset^* x$ for every set $x \in U_\alpha$. This iteration satisfies the $\kappa^{+}-c.c.$ So every family $F \subset [\kappa]^{\kappa^+}$ in $V^{\po_{\kappa^{++}}}$ 
  belongs already to $V^{\po_\alpha}$, for some $\alpha < \kappa^{++}$.  The fact that the rest of the iteration $\po_{\kappa^{++}}/\po_\alpha$ adds an $U_\alpha-$generating set, 
  implies that no family $F \subset [\kappa]^{\kappa}$ in $V^{\po_\alpha}$ is a splitting family in the final model $V^{\po_{\kappa^{++}}}$.
  
In our situation, we do not have a supercompact cardinal, and so it is unclear how to use generalized Mathias forcings.
However there is a natural replacement, the Radin forcing, which also produces generating sets.
Iteration of Radin forcing is problematic, but in many cases it is possible to avoid it.
So, suppose that $\kappa$ is a measurable of the Mitchell order $o(\kappa)\geq \kappa^{++}$, which is necessary by the result of
Zapletal \cite{Zapletal}. Let $\vec{U}=\l U_\alpha \mid \alpha<\kappa^{++}\r$ be a witnessing sequence of measures over $\kappa$.
The sequence is long enough to have repeat points (we assume $\GCH$).
Consider first applying Radin forcing with $\vec{U}$. This forcing is equivalent to the Radin forcing with the initial segment of $\vec{U}$, up to its first
repeat point. This implies that $2^\kappa$ will remain $\kappa^+$, and hence $s(\kappa)$ will not increase.\\

The next attempt will be to use the extender based Radin forcing, in order to blow up simultaneously the power of $\kappa$.
But how do we do this only with measures? It is possible to assume initially a bit more like ${\cal P}^2(\kappa)$--hyper-measurability, and forcing with 
a Mitchell increasing sequence of $\kappa^{++}-$ extenders. Although this version looks promising, it has some specific problems, for example it fails to introduce measurability of $\kappa$ in
some suitable intermediate extensions, which is a key of the construction (see section \ref{section - forcing}).
It turns out that the solution hides in the measures but requires a modification in the point of view. The basic idea is to perceive each measure $U_\alpha$ as an extender.
Namely, let $j_\alpha:V \to M_\alpha$ be the corresponding ultrapower embedding.
Derive a $(\kappa,\kappa+\alpha+1)$--extender $E_\alpha$ from $j_\alpha$, i.e.
\\$E_\alpha=\l E_\alpha(\beta)\mid \beta\leq \kappa+\alpha\r$, where
$X \in E_\alpha(\beta)$ iff $\beta \in j_\alpha(X)$.
Note that $\kappa$ is the single generator of $E_\alpha$ and all $ E_\alpha(\beta)$'s (with $\beta\geq \kappa$) are isomorphic to $U_\alpha$.
At the first glance, this replacement looks rather useless.
However, it turns out that there is a crucial difference
between the usual Prikry, Magidor, and Radin forcings and their extender based versions. This difference can be used to create a more complex repeat point structure. From a global point of view there are no repeat points in the (final) generic extension since the generic sequence added for $E_\alpha(\alpha)$ will allow us to separate $U_\alpha$ from the rest of the measures. On the other hand we will show that with certain restrictions, there are many subforcings which provide ``local repeat points". The local repeat points will be used to extend some $U_\alpha$ in the generic extensions by the subforcings. The splitting number argument is completed by proving that the rest of the extender based forcing adds a generating set to the extension of $U_\alpha$.
\\On the other hand, once one is interested in increasing the power of $\kappa$ only, then the number of generators plays the crucial role.
Namely, using a full extender or only its measures corresponding to the generators in the extender based Prikry forcing has the
same effect on the power of $\kappa$.

Let us show, for example, that the Prikry forcing and its extender based variation are not the same.
Let $U$ be a normal measure over $\kappa$, $j_U:V \to M_U$ the corresponding elementary embedding.
Define a $(\kappa,\kappa^+)$--extender $E=\l E(\beta) \mid \beta<\kappa^+\r$ derived from $j_U$:
$$X \in E(\beta) \text{ iff } \beta\in j_U(X).$$
Clearly, $U$ and $E$ have the same elementary embedding and the same ultrapower.
However the Prikry forcing ${\cal P}_U$ and the extender based  Prikry forcing ${\cal P}_E$
are not the same. Obviously, ${\cal P}_U$ is a natural subforcing of  ${\cal P}_E$, but the last forcing is richer.
Thus, let $\l t_\beta \mid \beta<\kappa^+\r$ be generic sequences added by a generic $G({\cal P}_E)$ of ${\cal P}_E$, i.e.
$t_\beta$ is a generic $\omega$--sequence for $E(\kappa,\beta)$.
Consider
$$A:=\{\beta<\kappa^+ \mid t_\beta(0)\not = t_\kappa(0) \}.$$
Then, obviously, $A \not \in V$. However, for every $\alpha<\kappa^+$,
$A \cap \alpha \in V$.
This implies that $A$ is not in a Prikry extension, since by \cite{G-Ka-Ko}
such extension cannot
add fresh subsets to $\kappa^+$.

So, in general, a restriction of an extender to the supremum of its generators may produce a weaker forcing than the forcing with the full extender.

Carmi Merimovich in \cite{carmi} introduced a very general setting for dealing with the extender based Magidor and Radin forcings.
The forcing used here will fit nicely his framework. We assume a familiarity with Merimovich's paper \cite{carmi} and will follow his notation.\\

\noindent \textbf{Acknowledgements:}\\
The authors are grateful to Ralf Schindler and the referee, for pointing out several errors in an earlier version
of this paper. The second author would like to thank J. Cummings and S. Friedman for stating to him the question about the consistency of the splitting number.

\section{Forcing $s(\kappa) = \lambda$ from $o(\kappa) = \lambda$}\label{section - forcing}

In this section we prove Theorem \ref{theorem - forcing}.
Let $\kappa, \lambda$ be regular cardinals such that $\kappa^+ < \lambda$ and $o(\kappa) = \lambda$.
\\Fix a Mitchell increasing sequence of extenders
 $\E = \l E_\alpha \mid \alpha < \lambda\r$ such that for every $\alpha < \lambda$

  \begin{enumerate}
    \item $\E \upharpoonright \alpha \in $Ult$(V,E_\alpha)$,\\
    where $\E \upharpoonright \alpha=\l E_\beta \mid \beta<\alpha \r$,
    \item $E_\alpha$ is a $(\kappa, \kappa+\sigma(E_\alpha))$--extender, for some $\sigma(E_\alpha),\alpha < \sigma(E_\alpha) < \lambda$.
  \end{enumerate}

  The first non-trivial case is $\lambda=\kappa^{++}$. We fix a Mitchell increasing sequence of measures $\l U_\alpha \mid \alpha<\kappa^{++}\r$
  on $\kappa$ and derive extenders $E_\alpha $ from the ultrapower embeddings
  $j_\alpha:V \to M_\alpha$ by $U_\alpha$'s.
  The simplest is to take $\sigma(E_\alpha)=\alpha+1$.

Following \cite{carmi}, we denote the Magidor-Radin extender based forcing associated with $\E$ by $\rad$.

 We will argue that $V^{\rad}$ satisfies $s(\kappa)=\lambda$.
 
 The outline of this argument is similar to the construction of S. Kamo, as sketched in
   \cite{Zapletal}. Every small family $F \subset [\kappa]^{\kappa}$ in a generic extension $V^{\po_{\E,\lambda}}$ is contained in a generic extension of a sub-forcing of $\po'$ of $\po_{\E,\lambda}$ for which   $\kappa$ is measurable in $V^{\po'}$, and the rest of the  forcing  $\po_{\E,\lambda}/\po'$ adds a generating set to some $V^{\po'}$ measure on $\kappa$. The sub-forcings we will use are the restrictions of $\rad$ to some suitable models $N \elem H_\theta$ for a sufficiently large regular cardinal $\theta$.

 Let $\{ \name{X}_i \mid i < \tau\}$ be a sequence of $\tau < \lambda$ many nice $\rad-$names of subsets of $\kappa$. Since $\rad$ satisfies the $\kappa^{++}.c.c$, we may find an elementary substructure $N \elem H_\theta$ for some sufficiently large regular $\theta$, which satisfies $\kappa^{+}\subset N$, $\rad,\E, \{ \name{X}_i \mid i < \tau\} \in N$, $|N| < \lambda$, and such that every $\name{X}_i$ is a $\rad \cap N$ name.
 
The key of the argument is that it is possible for $N \cap \rad$ to be a sub-forcing of $\rad$ by which $\kappa$ remains measurable and the complement forcing $\rad/(N \cap \rad)$ adds a generating set to some measure on $\kappa$. More precisely, assuming that $N \cap \lambda = \delta < \lambda$, we prove that $N \cap \rad$ is isomorphic to the extender based poset $\po_{\E\uhr\delta, \delta}$ associated to the restricted sequences $\E\uhr\delta = \{ E_\alpha \mid \alpha < \delta\}$.  We then apply a repeat point argument to prove that the $E_\delta-$normal measure $E_\delta(\kappa)$ extends to a measure $U_\delta$ in $V^{N \cap \rad}$, and prove that the completion poset adds a generating set $k_\delta$ to this measure.

\subsection{The sub-forcing $N \cap \rad$}
Let us fix some sufficiently large regular cardinal $\theta$ such that $\rad \in H_\theta$. Throughout this section we shall consider elementary substructures $N \elem H_\theta$ which satisfy:
\begin{itemize}
\item $|N| < \lambda$,
\item $\rad \in N$,
\item ${}^\kappa N \subset N$,
\item $N \cap \lambda = \delta \in \lambda$.
\end{itemize}
This is possible since we assume $\lambda$ is not a successor of a singular cardinal of cofinality $\leq \kappa$ and $\GCH$. 

\begin{lemma}\label{lemma - subforcing}
The poset $N \cap \rad$ is a sub-forcing of $\rad$, 
that is, the inclusion map of $N\cap \rad$ in $\rad$ is a complete embedding.
\end{lemma}

\begin{proof}
It is clear that for $p,q \in N \cap \rad$, $p \leq_{\rad} q$ if and only if $p \leq_{N \cap \rad} q$. Since $\rad \in N$, it is also clear that $p,q$ are incompatible in $\rad$ if and only if they are incompatible in $N \cap \rad$.
For every $p \in \rad$ define $q = p\uhr N \in \rad \cap N$ as follows:
First consider a condition $p = p_\rar = \la f,A\ra$ which consists only of its top part.  Denote $d = dom(f)$, then $A$ is a $d-$tree. For every $\nu \in OB(d)$ the restricted function $\nu\uhr N$ belongs to $OB(d\cap N)$. Define a $d \cap N-$ tree, $A\uhr N = \{ \la \nu_0\uhr N,...,\nu_n\uhr N \ra \mid \la \nu_0,...,\nu_n\ra \in A\ra$. Now set $p\uhr N = \la f \uhr N, A \uhr N\ra$.

For a general condition $p = p_\lar \fr p_\rar \in \rad$, then set $p\uhr N = p_\lar \fr (p_\rar\uhr N)$. It is straightforward to verify that for every $p \in \rad$, if $q' \in N$ is an extension of $p \uhr N$ then $q'$ is compatible with $p$. Therefore $N \cap \rad \subset \rad$ is a sub-forcing.
\end{proof}

Therefore, for every $V-$generic set $G \subset \rad$ we have that $G \cap N$ is $V-$generic for the poset $N \cap \rad$.

We would now like to show that $N \cap \rad$ is isomorphic to the poset $\po_{\E\uhr \delta, \delta}$ which is the Magidor-Radin forcing associated with the restricted sequence $\E \uhr \delta = \{E_\alpha \mid \alpha < \delta\}$. Writing that $\po_{\E\uhr\delta,\delta}$ is the forcing associated with $\E\uhr\delta$ entails a nontrivial statement that $\delta = \sup_{\alpha < \delta}j_{E_\alpha}(\kappa)$.  \footnote{see section 4 in \cite{carmi}.} Our assumption that $\alpha > \sigma(E_\alpha)$ for every $\alpha < \lambda$ implies that $\delta \geq \sup_{\alpha < \delta}j_{E_\alpha}(\kappa)$. Also, for every $\alpha < \lambda$ we assume $\sigma(E_\alpha) < \lambda$. Since $\lambda$ is a regular cardinal $j_{E_\alpha}(\kappa) < \lambda$, and if $\alpha < \delta = N \cap \lambda$ then $E_\alpha \in N$, so $j_{E_\alpha}(\kappa) \in N \cap \lambda = \delta$.

Another important consequences of $\delta = N \cap \lambda$ is the fact it is a local repeat point.
\begin{definition}
$\rho$ is a local repeat point of $\E$ if for every $x \in [\rho]^{\leq \kappa}$, letting $d = \{\ola \in \D \mid \alpha \in x\}$,
$$\bigcap_{\alpha < \rho}E_\alpha(d) = \bigcap_{\alpha < o(\E)}E_\alpha(d).$$
\end{definition}

\begin{lemma}\label{lemma - rp}
$\delta = N \cap \lambda$ is a local repeat point.
\end{lemma}
\begin{proof}
Take $x \in [\delta]^{\leq \kappa}$ and let  $d = \{\ola \in \D \mid \alpha \in x\}$. Then $d \in N$ since ${}^\kappa N \subset N$ so $\la E_\alpha(d) \mid \alpha < \lambda\ra \in N$ as well. Each $E_\alpha(d)$ measures $V_\kappa$, and since $\kappa^+ < \lambda$ and $\lambda$ is regular there exists some $\tau < \lambda$ such that
$$\bigcap_{\alpha < \tau}E_\alpha(d) = \bigcap_{\alpha < o(\E)}E_\alpha(d).$$
The elementarity of $N$ implies that there exists such $\tau < \lambda$ in $N$. Hence $\tau < \delta$ and the result follows.
\end{proof}

\begin{proposition}
$N \cap \rad$ is isomorphic to $\po_{\E\uhr\delta,\delta}$.
\end{proposition}
\begin{proof}
We shall compare the structures of conditions $p \in N \cap \rad$ with conditions $q \in \po_{\E\uhr\delta,\delta}$.

By making a small abuse of notation, let us denote the set
$\{\ol{\alpha} \in \D \mid \alpha < \delta\}$\footnote{see \cite{carmi} 4.2 for the definition of $\D$} by $\D \cap \delta$. Then $dom(f^{p_\rar}) \subset \D\cap\delta$ for every $p \in N \cap \rad$. 
We claim that the property $dom(f^{p_\rar}) \subset \D\cap\delta$ actually characterizes the conditions of $\rad$ which belong to $N$. First note that for any $p \in \rad$ we have $p_\lar \in V_\kappa \subset N$. 
Second, since $|f^{p_\rar}| \leq \kappa$, $rng(f^{p_\rar}) \subset V_\kappa$, and ${}^{\kappa}N \subset N$, we get that $dom(f^{p_\rar}) \subset \D\cap\delta$ implies $f^{p_\rar} \in N$. Finally, if $f^{p_\rar} \in N$ then every $\nu \in OB(dom(f^{p_\rar}))$ must be a member of $N$ as well and as $|A^{p_\rar}| = \kappa$ we find that $A^{p_\rar} \in N$.

Let us now consider this with the structure of conditions $q \in \po_{\E\uhr\delta,\delta}$. First note that the extender sequences
$$\ol{\alpha} = \la \alpha \ra \fr \la E_\xi \mid \xi < o(\ol{E}), \alpha < j_{E_\xi}(\kappa)\ra$$
appearing in the components of conditions $p \in \rad$ \footnote{these include $d= dom(f^{p_\rar}) $ and $dom(\nu)$ for $\nu \in OB(d)$}
are now replaced by a shorter extender sequences
$$\ol{\alpha}\uhr\delta = \la \alpha \ra \fr \la E_\xi \mid \xi < \delta, \alpha < j_{E_\xi}(\kappa)\ra.$$
So the base set used in the domains of functions appearing in $\po_{\E\uhr\delta,\delta}$ is
$$\D\uhr \delta = \{\ola\uhr\delta \mid \alpha < \delta\},$$
functions $f' \in \po_{\E\uhr\delta,\delta}^*$ have domain $d' \in  [\D\uhr
\delta]^{\leq \kappa}$, $f' : d' \to \R^{<\omega}$.
Objects $\nu' \in OB(d')$, measures $E_\xi(d')$, $\xi < o(\E\uhr\delta) = \delta\}$, and $d'-$trees $T \subset [OB(d')]^{<\omega}$ are the appropriate variants of the cut  down sequence $\E\uhr\delta$.
For every $d \subset [\D]^{\leq \kappa}$, let $d\uhr\delta = \{\ola\uhr\delta \mid \ola \in d\}$. We conclude that by replacing the sequences $\ola\uhr\delta \in \D\uhr\delta$ with $\ola \in \D \cap \delta$, we can therefore construct a simple translation map $T$ such that
\begin{itemize}
\item for $d \in [\D\cap \delta]^{<\kappa}$, $T$ maps function $f' : d\uhr\delta \to \R^{<\omega}$ to functions $f : d \to \R^{<\omega}$, by replacing every $\ola\uhr\delta \in dom(f')$ with $\ola \in dom(f)$.
\item $T$ maps objects $\nu' \in OB(d\uhr\delta)$ to objects $\nu \in  OB(d)$, by replacing every $\ola\uhr\delta \in dom(\nu')$ with $\ola \in dom(\nu)$.
\item By extending $T$ hereditarily, $T$ maps $d\uhr\delta-$trees $A' \subset [OB(d\uhr\delta)]^{<\omega}$ to trees $A \subset [OB(d)]^{<\omega}$.
\end{itemize}
$T$ is clearly a bijection. We claim that $T$ maps $d\uhr\delta-$trees to $d-$trees. Following the definition of the measures $E_\alpha(d)$, it is clear that for every $\alpha < \delta$, $T$ maps $E_\alpha(d\uhr\delta)$ sets to $E_\alpha(d)$ sets, this implies that sets
$$X \in E(d\uhr\delta) = \bigcap_{\alpha < \delta}  E_\alpha(d\uhr\delta)$$
are mapped to sets
$$ T(X) \in \bigcap_{\alpha < \delta}  E_\alpha(d),$$
which by lemma \ref{lemma - rp} are members of $E(d)$. It is clear that in $Y \in E(d)$ we have $T^{-1}(Y) \in \bigcap_{\alpha < \delta} E_\alpha(d\uhr\delta)$. It follows that we can extend $T$ to a bijection from  $\po_{\E\uhr\delta,\delta}$ to $N \cap \rad$:
first, for $q_\rar = \la f',A'\ra \in \po_{\E\uhr\delta,\delta}^\rar$ let $T(q_\rar) = \la T(f'),T(A')\ra$. Then for general $q = q_\lar\fr q_\rar \in \po_{\E\uhr\delta,\delta}$ set $T(q) = q_\lar \fr T(q_\rar)$. Obviously $T$ respects $\leq,\leq^*$ so $T$ is an isomorphisms of Prikry type forcings.
\end{proof}

\noindent Our next goal is to show that the normal ultrafilter $E_\delta(\kappa) = \{X \subset \kappa \mid \kappa \in j_{E_\delta}(X)\}$ extends to a normal ultrafilter in a generic extension by $N \cap \rad$. We will apply a variant of the repeat point arguments in Section 5 of \cite{carmi} to our situation in which $\delta$ is a local repeat point of $\E$.
Let $j_{E_\delta} : V \to M_\delta \cong Ult(V,E_\delta)$ be the $E_\delta$ induced ultrapower. Since $\E$ is a Mitchell increasing sequence of extenders then $\E\uhr\delta$ is the sequence of extenders which appears on $\kappa$ in $M_\delta$. Moreover $\E\uhr\delta$ is used to generate the measure $E_\delta(d)$ for every $d \in [\D]^{\leq \kappa}$. More precisely \cite{carmi} defines
$$mc_\delta(d) = \{\la j_{E_\delta}(\ola), R_\delta(\ola)\ra  \mid \ola \in d, \alpha < j_{E_\delta}(\kappa)\}$$
where $R_\delta(\ola)$ corresponds to an end segment of $\ola$ and is given by
$$ R_\delta(\ola) = \la \alpha \ra \fr \{E_\tau \mid \tau < \delta, \alpha < j_{E_\tau}(\kappa)\}.$$
The measure $E_\delta(d)$ is defined by
$X \in E_\delta(d)$ if and only if $mc_\delta(d) \in j_{E_\delta}(X)$.

Let $p \in \rad$ and consider the end extension $j_{E_\delta}(p)_{\la mc_\delta(d)}$ of $j_{E_\delta}(p)$ in $j_{E_\delta}(\rad)$. As  $R_\delta(\olk) = \la \kappa \ra \fr \E\uhr\delta$ it follows that
$${j_{E_\delta}(p)_{\la mc_\delta(d)\ra}}_\lar \in \po_{\E\uhr\delta,\delta}.\footnote{We use the fact that $\po_{\E\uhr\delta,\delta}^{M_\delta} = \po_{\E\uhr\delta,\delta}$}$$
Now assuming that $p \in N \cap \rad$, it is  straight forward to verify that
$$T({j_{E_\delta}(p)_{\la mc_\delta(d)\ra}}_\lar) = p.$$
We are now ready to prove that $E_\delta(\kappa)$ extends in a generic extension by $N \cap \rad$. We first need the following preliminary lemma.

\begin{lemma}\label{lem5}
Let $p \in N \cap \rad$ and $\name{X}$ be such that $p \force \name{X} \subset \kappa$, then there exists an extension $q \leq^* p$ such that
$$j_{E_\delta}(q)_{mc_\delta(q_\rar)} \dec \can{\kappa} \in j_{E_\delta}(\name{X}).$$
\end{lemma}

\begin{proof}
Choose an elementary sub model  $N^* \elem N$ such that $|N^*| = \kappa$, ${}^{<\kappa}N^* \subset N^*$, $N^* \cap \kappa^+ \in \kappa^+$, and $\name{X},p, N \cap \rad \in N^*$. The collection of all $\po_{\E\uhr\delta,\delta}$ dense sets in $N^*$ has cardinality $\kappa$. Using the fact $\leq^*$ for $\rad^*$ is $\kappa^+-$closed we can construct a direct extension $f^* \leq^* f^p$, and an $f^*-$tree $A^*$ such that
\begin{itemize}
\item $p^* = \la f^*,A^*\ra \leq^* p_\rar$,
\item for every $\la \nu \ra \in A^*$ and $D \in N$, if $D$ is dense open (in $\rad^*$) below $f^p_{\vec{\nu}\uhr dom(f^p)}$ then $f^*_{\vec{\nu}} \in D$.
\end{itemize}
Denote $dom(f)$ by $d$ and $\la f^*,A^*\ra$ by $p^*$. The construction of $p^*$ can be carried out inside $N$ so we may assume $p^* \in N$. For every $\la \nu \ra \in A^*$ let
$$ \displaylines{\quad  D_{\la\nu\ra}  = \{ g \leq^* f_{\la \nu \uhr d\ra} \mid \exists q_0,q_1,B, \text{ s.t. } q_1 \leq^* ({p_{\la \nu\ra}}_\lar)_\rar ,   \hfill \cr \hfill \ q_0 \leq^* p^*_\lar, \text{ and }
q_0 \fr q_1 \fr \la g,B\ra \dec \nu(\can{\kappa})_0 \in \name{X} \}. }
$$
Then $D_{\la \nu \ra}$ belongs to $N^*$ since ${p^*_{\la\nu\ra}}_\lar \in V_\kappa \subset N$ and is dense open (in $\rad^*$) below $f_{\la \nu\uhr d\ra}$ by the Prikry condition. Hence $f^*_{\la \nu \ra} \in D_{\la\nu\ra}$. Denote the components $q_0,q_1,B$ which witness $f^*_{\la \nu \ra} \in D_{\la\nu\ra}$ by $q_0(\nu),q_1(\nu),B(\nu)$ respectively. Since $q_0(\nu) \leq p_\lar$, there exists some fixed $q_0^*$ such that the set $\{ \nu \in Lev_0(A^*) \mid q_0^* = q_0(\nu) \} \in E_\delta(f^*)$.

Next define
$$q_1' = [q_1(\nu)]_{E_\delta(f^*)} = j_{E_\delta}(q_1)(mc_\delta(f^*)),$$
then $q_1' \leq^* {j_{E_\delta}(p^*)_{\la mc_\delta(f^*)\ra}}_\lar$.
Setting $q_1^* = T(q_1')$ then $q_1^* \in N \cap \rad$ is a direct extension of $p^*_\rar$.
Define $q^* = q_0^* \fr q_1^*$, then $q^*\leq^* p$. We need to reduce the tree $A^{q_1^*}$.
 Let $q$ be a direct extension of $q^*$ such that for
$A^{q_\rar}_{\la \nu \ra} \uhr dom(f^*) \subset B(\nu \uhr dom(f^*))$ for every $\nu \in Lev_0(A^{q_\rar})$.
Since $q_1'$ was defined via the $E_\delta(f^*)$ ultrapower, there exists a subset $Y \subset Lev_0(A^{q_\rar})$, $Y \in E_\delta(f^q)$ such that for all $\nu \in Y$,

$$q_{\la \nu \ra} \leq^* q_0(\nu) \fr q_1(\nu) \fr \la f^*_{\la \nu\uhr dom(f^*) \ra}, B(\nu\uhr dom(f^*))\ra$$ for every $\nu \in Y$.
\footnote{Note that we cannot in general  take $Y=Lev_0(A^{q_\rar})$, since 
$q_1'$ was defined by the $E_\delta(f^*)$ ultrapower, so the identification of ${{q_\rar}{\la \nu \ra}}_\lar$ with $q_1(\nu)$ may not hold for every $\nu \in Lev_0$ but only on some $E_\delta(f^*)$ set.}
Therefore $$j_{E_\delta}(q)_{mc_\delta(q_\rar)} \dec \can{\kappa} \in j_{E_\delta}(\name{X}).$$
\end{proof}

We are now ready to define the extension $E_\delta(\kappa)$. Let $G_{N} \subset N \cap \rad$ be $V-$generic filter.
\begin{definition}\label{definition - extended measure}
In $V[G_N]$ define $U_\delta \subset \power(\kappa)$ as follows:
For every $X \subset \kappa$ in $V[G_N]$, $X \in  U_\delta$ if and only if there exists some $p \in G_{N}$ such that
$$j_{E_\delta}(p)_{\la mc_\delta(p_\rar)\ra} \force \can{\kappa} \in j_{E_\delta}(\name{X}).$$
\end{definition}

Note that $p \force \name{X} \in \name{U_\delta}$ does not necessary imply that
$j_{E_\delta}(p)_{\la mc_\delta(p_\rar)\ra} \force \can{\kappa} \in j_{E_\delta}(\name{X}).$

\begin{proposition}\label{prop7}
$U_\delta$ is a $\kappa-$complete normal ultrafilter on $\kappa$ in $V[G_N]$. Furthermore it extends $E_\delta(\kappa)$.
\end{proposition}

\begin{proof}
We start by verifying that for every $p \in N \cap \rad $ with $p \force \name{X} \in \name{U_\delta}$, there is
a direct extension $p^* \leq^* p$ such that $p^*_\lar = p _\lar$ and $$j_{E_\delta}(p^*)_{\la mc_\delta(p^*_\rar) \ra} \force \can{\kappa} \in j_{E_\delta}(\name{X}).$$
First note that if $p \force \name{X} \in \name{U_\delta}$ and $j_{E_\delta}(p)_{\la mc_\delta(p) \ra} \dec \can{\kappa} \in j_{E_\delta}(\name{X})$ then $j_{E_\delta}(p)_{\la mc_\delta(p) \ra}$ must force $``\can{\kappa} \in j_{E_\delta}(\name{X})"$. Note that the number of possible extensions $p_\lar$ is less than $\kappa$. Let $\la r_i \mid i < \tau\ra$ be an enumeration these extensions, then we can construct an $\leq^*$-decreasing sequence of extensions $\la t_i \mid  i < \tau \ra$ stronger than $p_\rar$ such that for each $i < \tau$, if there exists some $t \leq^* t_{i\rar}$ such that for $p' = r_i \fr t$, $$j_{E_\delta}(p')_{\la mc_\delta(p') \ra} \force \can{\kappa} \in j_{E_\delta}(\name{X})$$
then  $t_{i+1}$ is one of such $t$. Otherwise, $t_{i+1} = t_i$.  Let $p^*$ be an extension of $p$ such that $p^*_\lar = p_\lar$ and $p^*_\rar$ is a common direct extension of all $\la t_i \mid i < \tau\ra$. It is easily seen that $p^*$ meets our requirements, by Lemma \ref{lem5}.

\noindent We can now show that $U_\delta$ is a normal $\kappa-$complete ultrafilter on $\kappa$.
$U_\delta$ clearly extends $E_\delta(\kappa)$.
If $X \in U_\delta$ and $X \subset Y \subset \kappa$, then for some suitable names $\name{X},\name{Y}$ there exists some $p \in G_N$ such that $p \force \name{X} \subset \name{Y}$ and $$j_{E_\delta}(p)_{\la mc_\delta(p_\rar)\ra} \force \can{\kappa} \in j_{E_\delta}(\name{X}).$$
Hence
$$j_{E_\delta}(p)_{\la mc_\delta(p_\rar)\ra} \force \can{\kappa} \in j_{E_\delta}(\name{Y}).$$
Next, let $p \in G_{N}$ be a condition forcing $``\la \name{X}_i \mid i < \kappa \ra \subset \name{U_\delta}"$. Using the observation in the beginning of this proof, we can construct a decreasing sequence of direct extensions below $p_\rar$ such that
$$j_{E_\delta}(p_i)_{mc_\delta(p_i)} \force \can{\kappa} \in j_{E_\delta}(\name{X}).$$
Now set $f^* = \bigcup_{i<\kappa}f^{p_i}$ and construct an $f^*-$tree $A^*$:
First set
$$Lev_0(A^*) = \{ \nu \in OB(f^*) \mid \forall i < \nu(\olk)_0 \quad \nu\uhr dom(f^{p_i}) \in A^{p_{i\rar}}\},$$
then, for every $i < \kappa$ define
$$A_i = \{ \vec{\nu} \in OB(f^*)^{<\omega} \mid \vec{\nu}\uhr dom(f^{p_i}) \in A^{p_{i\rar}}\}$$
and set
$$A^*_{\la \nu \ra} = \bigcap_{i < \nu(\olk)_0} A_i.$$
Then $p^* = p_\lar \fr \la f^* , A^* \ra$ force $\Delta_{i < \kappa} X_i \in U_\delta$.
\end{proof}

\begin{remark}
We point out the differences  between the local repeat point of $\E$ and (global) repeat point as defined in Section 5 of \cite{carmi}. If $\delta < o(\E)$ is a (global) repeat point of $\E$ then all the normal measures $\{ E_\gamma(\kappa) \mid \delta \leq \gamma < o(\E)\}$ extend to measures in the $\rad-$generic extension. Here in the local case, the argument is given for sub-forcing extensions $V^{\po_{\E\uhr\delta,\delta}}$, and $E_\delta(\kappa)$ is the only normal measure between those appearing in $\E$ which extends. For suitable $\delta' > \delta$ (that is $\delta' = N' \cap \rad$ for an appropriate structure $N'$) one needs to force with $\po_{\E\uhr\delta',\delta'}/\po_{\E\uhr\delta,\delta}$ over $V^{\po_{\E\uhr\delta,\delta}}$ in order to extend $E_{\delta'}(\kappa)$. 
In the next section we shall prove that forcing with $\po_{\E\uhr\delta',\delta'}/\po_{\E\uhr\delta,\delta}$ over $V^{\po_{\E\uhr\delta,\delta}}$ adds a generating set $k_\delta$ to the extension $U_\delta$
 of $E_\delta(\kappa)$. We can therefore show that none of the measures $\{E_\alpha(\kappa) \mid \alpha < o(\E)\}$ extends in the final model $V^{\rad}$.
Therefore the forcing $\rad$ can be thought of as an iteration of adding generating sets to measures. As we will see in the next section, the generating sets $k_\delta$ can be seen to come from differences between generic Magidor-Radin clubs. Under this interpretation $\rad$ ``hides" an iteration of club shooting.
\end{remark}

\subsection{Adding a generating set to $U_\delta$}
We now address the  forcing $\rad/N \cap \rad$. We shall prove that this forcing adds a generating set $k_\delta$ to the $V^{N \cap \rad}$ measure $U_\delta$ defined above. We need to add preliminary notations in order to define and work with the sets $k_\delta$.

Let $G \subset \rad$ be a $V-$generic set. For every $\alpha < \lambda$  the following sets where defined in \cite{carmi}:
\begin{itemize}
\item $G^{\alpha} = \bigcup\{f^{p_\rar}(\ola) \mid p \in G, \ola \in dom(f^{p_\rar}) \}$,
\item $C^{\alpha} = \{ \ol{\nu}_0 \mid \ol{\nu} \in G^{\ola} \}\subseteq \kappa$.
\end{itemize}
$C^{\kappa}$ is the well known Radin club associated with the Mitchell increasing sequence of measures $\{ E_\beta(\kappa) \mid \beta < \lambda\}$. The other generic sets $C^{\alpha}$, $\kappa < \alpha < \lambda$ are not clubs
and correspond to $\{ E_\beta(\alpha) \mid \beta < \lambda, \alpha< length(E_\beta)\}$. They can be associated with a filtration of $C^\kappa$ by clubs.

\begin{ddefinition}\label{definition - preliminary generic data}
Let $\tau \in C^{\kappa}$ be an ordinal in the generic Magidor-Radin club. Then there are $p\in G$ and $\nu \in Lev_0(A^{p_\rar})$ such that $\tau = \nu(\olk)_0$ and $p_{\la \nu \ra} \in G$.  Define
\begin{itemize}
\item $o^G(\tau) = o(\nu(\olk))$, for any (every) $p, \nu$ as above.
\\Recall that $\nu(\olk)=\l \tau,e_0,...,e_\xi,...\r$ ($\xi<\mu$) and $\mu$ is called $o(\nu(\olk))$, i.e. the
length of the sequence of extenders on $\tau$.
\item If $\alpha < \lambda$ is an ordinal for which $\ola \in dom(\nu)$ for some $p,\nu$ as above, then define $f_\alpha(\tau) = \nu(\ola)_0$.
\\Namely, it is the ordinal which corresponds to $\alpha$ ($C_\alpha$) over the level $\tau$.
\end{itemize}
\end{ddefinition}

\begin{ddefinition} For every $\alpha < \lambda$ set
$$k_\alpha = \{ \tau \in C^{\olk} \mid \tau \in dom(f_\alpha) \text{ and } o^G(\tau) = f_\alpha(\tau) \}.$$
\end{ddefinition}

\begin{remark} Let $p \in \rad$ and $\alpha < \lambda$ for which $\ola \in dom(f^{p_\rar})$. By examining the definition of \cite{carmi} the following can be easily verified:
\begin{enumerate}
\item $\{ \la \nu \ra \in A^{p_\rar} \mid \ola \in dom(\nu) \} \in E_\gamma(p_\rar)$ for every $\gamma \geq \alpha$.

\item Let
$$W_\alpha = \{ \nu \mid \nu \in OB(d) \text{ for some } d, \ola \in dom(\nu), \text{ and } \nu(\ola)_0 = o(\nu(\olk))\},$$ then
$$Lev_0(A^{p_\rar}) \cap W_\alpha \in E_\alpha(f^{p_\rar}) \setminus \bigcap_{\beta \neq \alpha}E_\beta(f^{p_\rar}).\footnote{see also \cite{carmi}, definition 4.9 and  $X_=,X_>,X_<$}
$$

\item $p \force \name{k_\alpha} \subset^* \{\nu(\olk)_0 \mid \nu \in Lev_0(A^{p_\rar}) \cap W_\alpha\}$. \footnote{Here  $Y \subset^* Y'$ if $Y' \setminus Y$ is bounded in $\kappa$.}

\end{enumerate}
\end{remark}

${}$ \newline
Let $N \elem H_\theta$, and $\delta = N \cap \lambda$ be as in the previous section. By lemma \ref{lemma - subforcing}, the set $G_{N} = G \cap N$ is a $V-$generic for $N \cap \rad$.
Let $U_\delta$ be the measure on $\kappa$ extending $E_\delta(\kappa)$ in $V[G_N]$. We claim

\begin{proposition}\label{proposition - generating sets}
$k_\delta$ is a $U_{\delta}-$generating set, i.e. for every $X \subset \kappa$ in $V[G_{N}]$ if $X \in U_\delta$ then $k_\delta \subset^* X$ in $V[G]$.
\end{proposition}
\begin{proof}
Let $\name{X}$ be a name for $X$ in $V[G_{N}]$ and let $p \in G_N$ such that $p \force_{\rad\cap N^*} \name{X} \in \name{U_\delta}$. Let us show that $p$ has an extension $q \leq p$, $q \force_{\rad} \name{k}_\delta \subset^* \name{X}$.
By 5.8 in \cite{carmi} or by the argument of Proposition \ref{prop7}, we may assume  $$j_{E_\delta}(p)_{\la mc_\delta(p_\rar)\ra} \force \can{\kappa} \in j_{E_\delta}(\name{X}).$$
Define
$$X' = \{ \nu  \in Lev_0(A^{p_\rar}) \mid p_{\la \nu \ra} \force \can{\nu(\olk)_0} \in \name{X} \}.$$
Then $X' \in E_\delta(p_\rar)$. Let $p^*$ be a direct extension of $p$ in $\rad$ obtained by adding $\old$ to $dom(f^{p_\rar})$.\\Set
$$X^* = \{\nu  \in Lev_0(A^{p^*_\rar}) \mid p^*_{\la \nu \ra} \force \can{\nu(\olk)_0} \in \name{X} \}.$$
Clearly $X^* \in E_\delta(p^*_\rar)$. Let $p^{**}$ be the strong Prikry extension of $p^*$ obtained by reducing $Lev_0(A^{p^*})$ to $\left(Lev_0(A^{p^*}) \setminus W_\delta \right) \cup \left( Lev_0(A^{p^*}) \cap W_\delta \cap X^* \right)$. We claim that $p^{**} \force \name{k}_\delta \subset^* \name{X}$. Since
$p^{**} \force \name{k_\delta} \subset^* \{\nu(\olk)_0 \mid \nu \in Lev_0(A^{p_\rar}) \cap W_\delta\}$,
it is sufficient to consider elements $\nu(\olk)_0$ for $\nu \in Lev_0(A^{p^{**}})$.
Let $q \geq p^{**}$  such that $q \force \nu(\olk)_0 \in \name{k_\delta}$, then $q \force p^{**}_{\la \nu \ra} \in \name{G}$, but $p^{**}_{\la \nu \ra} \force \nu(\olk)_0 \in \name{X}$ as $\nu \in X^{*}$.
\end{proof}

We can now deduce the main result of this section.
\begin{theorem}
For every $V-$generic filter $G \subset \rad$,
$\kappa$ is a regular and $s(\kappa) = \lambda$ in $V[G]$.
\end{theorem}
\begin{proof}
The fact $\kappa$ remains regular and $2^\kappa \leq \lambda$ in $V[G]$ is obtained in \cite{carmi}. Let $F \subset [\kappa]^\kappa$ be a family of size $|F| < \lambda$. Since no cardinals are collapsed in $\rad$ we can form a sequence of  names $\{ \name{X}_i \mid i < \eta \}$ for some $\eta < \lambda$ which are interpreted as $F$ in $V[G]$. Take an elementary submodel $N \elem H_\theta$ as described at the beginning of this section, such that $\{ \name{X}_i \mid i < \eta \} \subset N$, and denote $\delta = N \cap \lambda$.

Since $\rad$ satisfies $\kappa^{++}-c.c.$ and $\kappa^+ \subset N$, we may assume that  $\name{X}_i$ are $N \cap \rad$ names, so $X_i \in V[G \cap N]$ for every $i < \eta$.
$U_\delta$ is a measure on $\kappa$ in $V[G\cap N]$ and by Proposition \ref{proposition - generating sets}, $k_\delta \in V[G]$ is $U_\delta$ generating set. If follows that none of the sets $X_i$ splits $k_\delta$ hence $F$ cannot be a splitting family in $V[G]$. We conclude that $s(\kappa) = \lambda$.
\end{proof}

\section{From $s(\kappa) = \lambda$ to $o(\kappa) = \lambda$}\label{section - inner model}

We prove Theorem \ref{theorem - inner model}, which shows that the initial assumption $o(\kappa) = \lambda$ of our previous construction is the optimal one.
This generalizes the argument of \cite{Zapletal} which proves that $s(\kappa) = \kappa^{++}$ implies that there exists an inner model with a measurable cardinal $\alpha$ such that $o(\alpha) = \alpha^{++}$.
The argument of \cite{Zapletal} relies on the structure of the core model $\K$ below $o(\alpha) = \alpha^{++}$ which is a core model for a sequence of measures. The fact that all extenders are equivalent to measures is required in \cite{Zapletal} to argue that when $U$ is the normal measures derived from an  elementary embedding  $i : \K \to \K^*$, then $U$ does not belong to $\K^*$.
Unfortunately, this property fails in core models for sequences of extenders. The argument suggested here appeals to the ideas of \cite{gitik - extender}.

\begin{proof}{(Theorem \ref{theorem - inner model})}\\

We prove the Theorem in \textbf{two steps}: We first prove the result assuming
$\lambda$ is an $\omega_1$ closed cardinal (that is, $\alpha^{\omega_1} < \lambda$ for each $\alpha < \lambda$),
and then, use this result to show that if $\lambda > \kappa^+$ is not a successor of a singular cardinal of
cofinality $\leq \omega_1$, then $\lambda$ must be $\omega_1$ closed. \\

\textbf{Step I: \vspace{0.5pt}} Suppose that $\lambda$ is an $\omega_1$ closed cardinal. Assuming $\neg 0^{\P}$, we claim
that that $s(\kappa) \geq \lambda$ implies that $o(\kappa) = \lambda$ in the core model $\K = J^E$. 
Suppose otherwise. Since $o^{\K}(\kappa) < \lambda$, there exists some $\eta < \lambda$ such that $J_\eta^E$ is a mouse which includes all extenders on the sequence $E$ with critical point $\kappa$. 
Let us say that an iteration $i : J_\eta^E \to Y$ is mild in $\kappa$ if
$\kappa = crit(i)$ and no extender of $E \uhr\eta$ is used more than $\omega_1$ many times along this iteration.  Since $\lambda$ is a regular cardinal the set
\[ \{i(\kappa) \mid i  : J_\eta^E \to Y, \text{ is mild in } \kappa \}\]
is bounded by some $\tau < \lambda$.
Since $\kappa$ us inaccessible and $\lambda$ is $\omega_1$ closed, we can choose a sufficiently large regular cardinal $\theta > \lambda$, and find $N \elem H_\theta$ which satisfies:
\begin{enumerate}
\item $\tau + 1 \subset N$,
\item $|N| < \lambda$,
\item ${}^{\omega_1}N\subset N$,
\item $J_\eta^E \subset N$,
\item $\lambda \in N$.
\end{enumerate}
Let $N_0$ be the transitive collapse of $N$, then all but the last property are valid in $N_0$, and
\[N_0 \models \kappa \text{ is inaccessible and }2^\kappa > \tau.\]
Since $|\power(\kappa) \cap N_0| < \lambda$, this set cannot be a splitting family. Let $a \in \power(\kappa)$ be a witness, that is either $|a \setminus x| < \kappa$ or $|a \cap x| < \kappa$ for every $x \in \power(\kappa) \cap N_0$.
The induced $U_a = \{ x \in  \power(\kappa) \cap N_0 \mid |a \setminus x| < \kappa\}$ is a $\kappa$ complete nonprincipal $N_0-$ultrafilter.
Hence the structure $Ult(N_0,U_a) = (N_0 \cap {}^\kappa N_0)/U_a$ is well founded. Denote its transitive collapse by $N^*$. $N_0$ satisfies sufficient fraction of set theory to apply  
\L o\'{s} theorem and obtain an elementary embedding $i^* : N_0 \to N^*$.
We conclude
\begin{enumerate}
\item $V_\kappa^{\K} = V_\kappa \cap J_\eta^E \subset N^*$,
\item $N^* \models i^*(\kappa) > 2^\kappa > \tau$,
\item ${}^{\omega_1 }N^* \subset N^*$.
\end{enumerate}
We now appeal to inner model theory, as presented in \cite{Zeman}, and to the proof of Theorem 1.1 in \cite{gitik - extender}.
First, both $N_0$ and $N^*$ satisfy sufficient fraction of set theory to define their core models $\K(N_0)$, and $\K(N^*)$, and prove the appropriate covering Lemma used
in \cite{gitik - extender}. Let $i = i^*\uhr \K(N_0)$, then the definability of the core model implies that $i: \K(N_0) \to \K(N)$ is elementary.
Since $J_\eta^E \in  N_0$, it follows that $J_\eta^E = J_{\eta}^{E^{N_0}}$, that is $J_\eta^E$ is an initial segment of $\K(N_0)$. Furthermore, as $\kappa+1 \subset N$, then
$N$ witness that $cp(E_\alpha) > \kappa$ for all $\alpha \geq \eta$. Hence, the same is true in $\K(N_0)$.

We claim that there exists a normal iteration $\pi : \K(N_0) \to K'$ of $\K(N_0)$,  obtained by ultrapowers via extenders which originate in $J_\eta^E$, and that
$i = k \circ \pi$, where $k : \K' \to \K(N^*)$ satisfies $cp(k) > \pi(\kappa)$.
For this, consider the coiteration of $\K(N_0)$ with $\K(N^*)$.
Let us denote $\K(N_0)$ by $\K^0$, $\K(N^*)$ by $\K^*$, their coiterands by $\K^0_i$ and $\K^*_i$ respectively, and their iteration maps by $\pi^{\K_0}_{i,j}$ and $\pi^{\K^*}_{i,j}$, 
for $i<j$ below the length of the coiteration.
$J_\eta^E = \K^0\uhr\eta$ is an initial segment of the core model $\K$ (i.e. it is incompressible). 
The arguments of sections 7.4,8.3 of \cite{Zeman} imply that $\K^*$ does not move along in an initial segement of the coiteration,  as long as the first point of disagreement between $\K^0_i$ and $\K^*$ is below $\pi^{\K_0}_{0,i}(\eta)$.
It follows that if $\theta$ is the first index $i$ of the coiteration  in which $\K^0_i$ and $\K^*_i$ agree above $\pi^{\K^0}_i(\kappa)+1$, then $\K^*_i = \K^*$. 
We set $\K' = \K^0_\theta$ and $\pi = \pi^{\K_0}_\theta : \K(N_0) \to \K'$. 

Hence, every  $x \in \K'$ is of the form $\pi(f)(\xi_1,..,\xi_n)$, where $n < \omega$,  $f : \kappa^n \to \K(N_0)$ is a function in $\K(N_0)$, and
$\xi_1,..,\xi_n \leq \pi(\kappa)$. 
We then define $k : \K' \to \K(N^*)$ by sending $\pi(f)(\xi_1,..,\xi_n)$ as above, to $i(f)(\xi_1,..,\xi_n)$. It is standard to verify that
$k$ is well define, elementary, and $cp(k) > \pi(\kappa)$.
Finally, as ${}^{<\kappa }N^* \subset N^*$, we can apply the proof of Theorem 1.1 in \cite{gitik - extender} to $N^*$ with respect to  $\K(N^*)\uhr \pi(\kappa) = \K'\uhr\pi(\kappa)$, and $\pi$.
We conclude that $\pi\uhr J_\eta^E : J_\eta^E \to \pi(J_\eta^E)$ is mild. However this is absurd as $cp(k) > \pi(\kappa)$ so 
$\pi(\kappa) = k\circ \pi(\kappa) =  i(\kappa) = i^*(\kappa) > \tau$,
contradicting the choice of $\tau$. \\

\textbf{Step II: \vspace{0.5pt}} We now show that $\lambda > \kappa^+$ must be $\omega_1$ closed. 
Suppose otherwise, and let $\alpha < \lambda$ be the minimal cardinal such that
$\alpha^{\omega_1} \geq \lambda$. Clearly $\alpha$ must be a singular cardinal with $o(\alpha) \leq \omega_1$.
Since $s(\kappa) \geq \kappa^+$, $\kappa$ is strongly inaccessible and therefore $\alpha > \kappa$. 
It follows that the singular cardinal hypothesis fails at $\alpha$ (that is, $\alpha^{\cf(\alpha)} > \alpha^+ + 2^{\cf(\alpha)}$). Let $\beta \leq \alpha$ be the first singular cardinal above $\kappa$ in which the singular cardinal
hypothesis fails. By Silver's Theorem we have $\cf(\beta) = \omega$, therefore $\beta^\omega > \beta^+$.
We claim that the last implies that there are unbounded many $\gamma < \beta$ which are measurable cardinals
in $\K$. This will suffice to establish a contradiction since each $(\gamma^+)^V$
is $\omega_1$ closed by the minimality of $\beta > (\gamma^+)^V$, and as $s(\kappa) \geq (\gamma^+)^V$,
we conclude from the result above that $o^{\K}(\kappa) \geq (\gamma^+)^V \geq (\gamma^+)$. 
It follows that $\K$ contains an extender on $\kappa$ which overlaps a measure on $\gamma$, contradicting $\neg 0^{\P}$.
Finally, in order to show that $\beta$ is a limit of $\K-$measurable cardinals, we appeal to Shelah's $\pcf$ theory, and to Gitik's results on the failure of the singular cardinal hypothesis.
By Theorem A in Section 2 of \cite{Gitik - neg SCH}, it is sufficient to verify that $\pp(\beta) > \beta^+$.
To this end, we note that $\cov(\beta,\beta,\omega_1,2) > \beta^+$. Indeed if
$X \subset \power_\beta(\beta)$ is a covering family of $\power_{\omega_1}(\beta)$
then $|X| > \beta^+$. Otherwise, $\beta^\omega > \beta^+$ would imply
that there are $\beta^{++}$ many different sets in $\power_{\omega_1}(\beta)$ which are covered
by the same set $z \in X$. This is impossible since $|z| < \beta$, and $\power_{\omega_1}(|z|) < \beta$
by our choice of $\beta$. Next, by \cite{Shelah - Fur Car Arthm} (see also $(e)$ in page 446 of \cite{Shelah - Card Arthm})
we either have $\pp(\beta) = \cov(\beta,\beta,\omega_1,2) > \beta^+$, or $\pp(\beta) > \beta^+$. 
Either way, we get $\pp(\beta) > \beta^+$. 
\end{proof}

\subsection*{Open Questions}

Let us conclude with some questions:\\

{\bf Question 1}. What is the consistency strength of $s(\kappa) = \kappa^{+\omega+1}$?\\
The authors believe to have found a modification of the forcing argument in Section \ref{section - forcing}
which shows that the consistency strength is not greater than $o(\kappa) = \kappa^{+\omega+1}$, however it is not clear whether this is optimal.\\

{\bf Question 2}. What is the consistency strength of the statement that $\kappa$ is a measurable and $s(\kappa)=\kappa^{++}$?\\
In the model of Kamo, $\kappa$ remains a measurable (and even a supercompact).\\

{\bf Question 3}. Is it possible to have GCH below $\kappa$ and $s(\kappa)=\kappa^{+3}$?\\
Note that in our model for $s(\kappa)=\kappa^{+3}$, $2^\alpha=\alpha^{++}$ holds on a club below $\kappa$.\\

{\bf Question 4}. Is it possible $s(\kappa)=\lambda$ for a singular $\lambda$? \\
Note that this is known for $\kappa = \aleph_0$.

\raggedright

\end{document}